\documentclass[
11pt, 
a4paper]{article} 

\usepackage{amsmath,amssymb,amsthm,amsfonts,latexsym,graphicx,subfigure}\usepackage{url}

\newtheorem{thm}{Theorem}
\newtheorem{defn}{Definition}
\newtheorem{pro}{Proposition}
\newtheorem{lemma}{Lemma}
\newtheorem{remark}{Remark}
\newtheorem{corollary}{Corollary}
\newtheorem{example}{Example}
\newcommand{\R}{\mathbb{R}}
\newcommand{\supp}{\mathrm{supp}}

\allowdisplaybreaks[4]
\begin{document}

\title{Khintchine inequality on normed spaces and the application to Banach-Mazur distance}
\author{Xin Luo\footnote{Academy of Mathematics and Systems Science, Chinese Academy of Sciences, Beijing 100190, P.R. China. \;\;\;\;\;\;\;~~~\;\;\;\;\;\;\; Email addresses: xinluo@amss.ac.cn \; and \;  xinlnew@163.com (Xin Luo).}
\and Dong Zhang\footnote{ Max Planck Institute for Mathematics in the Sciences, Inselstr. 22, 04103 Leipzig, Germany. \;\;\;\;\;\;\;~~~\;\;\;\;\;\;\;\;\;\;\;\;\;\;\;\;\;\;\;\;\; \;\;\;\;\;\;\;\;\; Email addresses:  13699289001@163.com \; and \;  dzhang@mis.mpg.de (Dong Zhang).} 
}
\date{}

\maketitle

\begin{abstract}
We establish variant Khintchine inequalities on normed spaces of Hanner type and cotype,
in which the Rademacher distribution corresponding to classical Khintchine inequality is replaced by general symmetric distributions. 
The proof involves the $p$-barycenter 
and Birkhoff's ergodic theorem.  More importantly, by employing these Khintchine inequalities, we get some
 lower bounds 
for Banach-Mazur distance between $l^p$-ball and a general centrally symmetric convex body. 

\vspace{0.2cm}

\end{abstract}

\section{Introduction}
\label{intro}
 The classical Khintchine inequality states that for any $0< p <\infty$, there are constants $A_p$ and $B_p$ such that
$$
 A_p(\sum^n_{i=1}v^2_i)^\frac12\leq \left(E|\sum^n_{i=1}\epsilon_iv_i|^p\right)^\frac1p\leq B_p(\sum^n_{i=1}v^2_i)^\frac12
$$
 where $(v_1,\cdots,v_n)\in \R^n$, and $\{\epsilon_i\}^n_{i=1}$ is a sequence of independent random variables in Rademacher distribution.
 Khintchine inequality attracts much  
 attentions from various fields, particularly in probability, functional analysis and combinatorics (see
\cite{M17,PassSpektor18,AstashkinCurbera15,Astashkin14,Konig14,Alanis18,AstashkinCurbera14,Astashkin17}).
We specially refer the readers to the systematical works by Astashkin \cite{AstashkinCurbera15,Astashkin14,AstashkinCurbera14,Astashkin17,ASV00}.

The aims 
of this paper 
are two-fold: (I) extend Khintchine inequality to general normed spaces; (II) find more connections to other fields. 

For (I), we introduce a large family of normed spaces, namely, the Hanner type spaces (see Definition \ref{defn}), and then establish variant Khintchine inequalities on such spaces (see Theorem \ref{Khint-norm1}). 
Precisely, we replace the Rademacher random variables $\{\epsilon_i\}^n_{i=1}$ by general symmetric random variables, and we extend the real numbers $\{v_i\}^n_{i=1}$ to vectors in normed spaces with certain properties. The constants appearing in the upper and lower bounds are sharp in some special situations. It is worth noting that our proof involves the theory of barycenter and Birkhoff's theorem for measure preserving transformations.



For (II), we find that 
the generalized Khintchine inequalities in Theorem \ref{Khint-norm1} 
can be applied to the estimate of Banach-Mazur distance.
For two centrally symmetric convex bodies $K$ and $L$ in $\R^n$, the (multiplicative) Banach-Mazur distance of $K$ and $L$ is defined by
$$d(K,L):=\inf\{r\ge 1: L\subset T(K)\subset rL,T\in GL(\R^n)\},$$
in which $GL(\R^n)$ is 
the general linear group on $\R^n$. 
Banach-Mazur distance is an important topic in the fields of 
convex geometry and functional analysis.
There are several results on the Banach-Mazur distance of special convex bodies, such as cubes, balls and crosspolytopes \cite{FeiXue18,Kobos20}.  However, it is difficult to estimate the Banach-Mazur distance $d(K,L)$ in general, even though for some simple cases. 
As an instance, Daws, Johnson, Serre and Thurston discussed this problem in MathOverflow in 2010 \cite{mathoverflow}.
In the present paper, as shown in Theorems \ref{thm:2}, 
we provide several lower bounds for the Banach-Mazur distance $d(l^\infty,\|\cdot\|)$ between the $l^\infty$-norm and a general norm $\|\cdot\|$.   
The proof heavily depends on our extended Khintchine inequality (Theorem \ref{Khint-norm1}). 
We further show a lower bound estimation for $d(l^p,\|\cdot\|)$ (see Proposition \ref{pro:lp-general}). As a corollary, the
lower bound of $d(l^p,l^q)$ with $1\le p<2<q\le \infty$ is obtained, which partially answers the questions in MathOverflow \cite{mathoverflow}. As an example, we 
give an answer to a question on the case of $n$-dimensional cubes and crosspolytopes \cite{FeiXue18}. 

This paper is organised as follows. In Section \ref{sec:K-H}, we introduce the Hanner type and cotype for normed spaces, and propose a Khintchine inequality on such 
normed spaces (see Theorem \ref{Khint-norm1}).  
The proof of Theorem \ref{Khint-norm1} is given in Subsection \ref{subsec:proof}. While some preparatory works are shown in Subsection \ref{subsec:Lemmas} which may possess some independent interests.  In Section \ref{sec:BM}, we use our extended Khintchine inequality to get some lower bounds for Banach-Mazur distance (see Theorems \ref{thm:2}). 
An elementary but technical proof of an auxiliary sharp inequality (see Lemma \ref{lemma:P(n,k)}) is presented in Appendix.
\section{Khintchine-type inequalities on normed spaces} 
\label{sec:K-H}

Khinchine inequality is widely known in probability, and it is also frequently used in functional analysis.
In this section, we enlarge the scope of Khinchine inequality from the perspective of functional analysis.

\vspace{0.1cm}

\noindent \textbf{Basic setting}: Given a probability measure space $(X,\mu)$ and a normed space $(\mathbb{E},\|\cdot\|)$, for $f\in L^\infty(X,\mu)$ and $v=(v_1,\cdots,v_n)\in \mathbb{E}^n$,
we introduce the functional 
\begin{equation}\label{eq:Ipvf}
I_p(v,f)=\left(\int_{X^n}\left\|\sum_{i=1}^n f(x_i)v_{i}\right\|^pdx_1\cdots dx_n\right)^{\frac1p}
\end{equation}
where 
$n\in \mathbb{Z}_+$  and $p\ge1$. 

We further assume that there is a measurable involution $-$ $: X\rightarrow X$ satisfying $f(-x)=-f(x)$ for $x\in X$ a.e., and $\mu(A)=\mu(-A)$ for any measurable subset $A \subset X$.  In this setting, we call such $f$ an {\sl odd function} on $X$ (i.e., its distribution satisfies $\mu(f>c)=\mu(f<-c)$ for any $c\ge0$).

For example, $X$ can be chosen as a centrally symmetric set or an axially symmetric set
endowed with a probability measure,  and $f$ can be an odd function on $X$ in the classical sense. In fact, without loss of generality, 
we can assume $X=[-\frac12,\frac12]$ in this section.

%

\begin{defn} \label{defn}
Given $p>0$, $n\in\mathbb{Z}_+$ and a normed space $\mathbb{E}:=(\mathbb{E},\|\cdot\|)$, we say $\mathbb{E}$ (or $\|\cdot\|$) is of Hanner cotype $(p,n)$, if
for any $x_1,\cdots, x_n\in \mathbb{E}$, there is
$$
\sum_{(\epsilon_{1},\cdots,\epsilon_{n})\in \{-1,1\}^n}\left\|\sum^{n}_{i=1}\epsilon_i x_i\right\|^p\geq \sum_{(\epsilon_{1},\cdots,\epsilon_{n})\in \{-1,1\}^n}\left|\sum^{n}_{i=1}\epsilon_i \|x_i\|\right|^p.
$$
Similarly, we say $\mathbb{E}$ is of Hanner type $(p,n)$, if 
the above inequality is reversed.
\end{defn}

\begin{remark}
In \cite{KOT96,YTK06}, a normed space is said to be of Hanner cotype (resp., Hanner type) $p$, if it is  Hanner cotype (resp., Hanner type) $(p,n)$ for all $n\in\mathbb{Z}_+$.
It is known that  $L^p$ space is of Hanner cotype $p$ for $1\le p\le 2$, and of Hanner type $p$ for $p\ge 2$.

Moreover, a norm satisfying the Hlawka inequality
$$\|x\|+\|y\|+\|z\|+\|x+y+z\|\ge \|x+y\|+\|y+z\|+\|z+x\|$$
is of Hanner cotype $(1,3)$. Thus, besides $L^1$ norm, there are many other norms (like the norm induced by the support function of a zonoid) is of Hanner cotype $(1,3)$.
 Furthermore, by Witsenhausen's result \cite{W73}, we know that a hypermetric normed space is of Hanner cotype $(1,n)$ for any $n\in\mathbb{Z}_+$.
\end{remark}

In this paper, we use 
$A_p$ and $B_p$ to denote 
the minimum and maximum of the
set $$\left\{1, 2^{\frac 12-\frac 1p}, 2^{\frac 12}(\Gamma(\frac{p+1}{2})/\sqrt{\pi})^\frac 1p\right\},$$ in which $\Gamma(\cdot)$ is 
the standard gamma function.

We establish the generalized 
Khintchine inequalities on $I_p(\cdot,\cdot)$ as follows.

\begin{thm}\label{Khint-norm1}
Given a normed space $\mathbb{E}$, and an odd function $f\in L^{\infty}(X)$,  we have the following:

\noindent (1) If $\mathbb{E}$ is of Hanner cotype $(q,n)$ with $q\leq p$, then for any $v\in\mathbb{E}^n$, $$ I_p(v,f)\geq c_{f,p,q}\sqrt{\sum^n_{i=1}\|v_{i}\|^2}$$ 
 where  $c_{f,p,q}= A_q \sup\limits_{S\subset \{f>0\}}\min\{(2\mu(S))^{\frac1p-1},(2\mu(S))^{-\frac12}\}\|f|_{S\cup (-S)}\|_1$.

\noindent (2) If $\mathbb{E}$ is of Hanner type $(q,n)$ with $q\ge p$, then
 $$I_p(v,f)\leq C_{f,p,q}\sqrt{\sum^n_{i=1}\|v_{i}\|^2}$$  for any $v\in\mathbb{E}^n$,  
where $C_{f,p,q}= B_q\max\{\mu (\mathrm{supp}(f))^{\frac1p},\mu (\mathrm{supp}(f))^{\frac12}\}\|f\|_\infty.$ 

\noindent (3) If $\mathbb{E}$ is an $L^2$-space, then for any $v\in l^2(\mathbb{E})$,
 $$c_{f,p}\|v\|_2 \le I_p(v,f)\le C_{f,p} \|v\|_2$$
 where $c_{f,p}= A_p\max\{\mu (\mathrm{supp}(f))^{\frac1p-1},\mu (\mathrm{supp}(f))^{-\frac12}\}\|f\|_1$ and $C_{f,p}=B_p\max\{\mu (\mathrm{supp}(f))^{\frac1p},\mu (\mathrm{supp}(f))^{\frac12}\}\|f\|_\infty$.

\end{thm}
\begin{remark}
In the case of the real line $\R$, Astashkin showed a better estimate than Theorem \ref{Khint-norm1} (3) (see \cite{ASV00}).
\end{remark}
\begin{example}
 Let $(X,\mu)$ and $f$ be defined by $X=\{-1,1\}$, $\mu\{-1\}=\mu\{1\}=\frac12$, $f(-1)=-1$ and $f(1)=1$. Then Theorem \ref{Khint-norm1} implies $\forall v\in l^2(\R)$,
  $$A_p\|v\|_2\le I_p(v,f)=\left(\frac{1}{2^n}\left|\sum_{i=1}^n \epsilon_iv_{i}\right|^p\right)^{\frac1p}\le B_p\|v\|_2$$
which is nothing but the classical Khintchine inequality.
\end{example}
 We prove that $I_p(v,f)$ essentially 
 defines a norm of $v$ on $\mathbb{E}^n$ if $f$ is nonconstant, and also defines a norm of $f$ on $L^p(X,\mu)$ if $\sum_{i=1}^n v_i\ne 0$. Thus, we may call $I_p(\cdot,\cdot)$ a bi-norm form. To some extent, Theorem \ref{Khint-norm1} 
 not only enlarges the scope of Khinchine inequality, but also derives two comparable norms, $\|v\|_{p,f}:=I_p(v,f)$
and $\|v\|_{l^2(\mathbb{E})}:=\sqrt{\sum^n_{i=1}\|v_i\|^2}$. 

\subsection{Auxiliary Lemmas }
\label{subsec:Lemmas}

\label{sec:proof}
\begin{lemma}\label{lemma:P(n,k)}
Given $n,k\in \mathbb{Z}_+$, $k\le n$, $\alpha\ge 0$, then for any $x_i\ge0$, $i=1,2,\cdots,n$,
there is
\begin{equation}\label{eq:lemma-main}
\max\left\{\frac kn,\left(\frac kn\right)^{\alpha}\right\}  \ge \frac{\sum_{1\le i_1<\cdots<i_k\le n}\left(\sum_{j=1}^k x_{i_j}\right)^\alpha}{{n\choose k}\left(\sum_{i=1}^n x_i\right)^\alpha} \ge \min\left\{\frac kn,\left(\frac kn\right)^{\alpha}\right\}
\end{equation}
in which both the upper and lower bounds are sharp. 
Here $n\choose k$ appearing in \eqref{eq:lemma-main} is the combinatorial number representing $k$-combination of $n$-elements.
\end{lemma}

We present the technical proof of Lemma \ref{lemma:P(n,k)} in Appendix. Now, we concentrate on the properties of $I_p(v,f)$.

\begin{lemma}\label{hannerAB-P}
Given $S\subset X$ with $S\cap (-S)=\varnothing$, let $$f(x)=(1_S-1_{-S})(x):=\begin{cases}
1,&\text{ if } x\in S,\\
-1,&\text{ if } x\in -S,\\
0,&\text{ otherwise.} \\
\end{cases}$$ 

(1) If $\mathbb{E}$ is of Hanner cotype $(q,n)$ with $q\leq p$, then $$I_p(v,f)\geq A_q\min\{(2\mu(S))^\frac1p,(2\mu(S))^\frac12\}(\sum_{i=1}^n \|v_{i}\|^2)^{\frac{1}{2}}.$$

(2) If $\mathbb{E}$ is of Hanner type $(q,n)$ with $q\geq p$, then $$I_p(v,f)\leq B_q\max\{(2\mu(S))^\frac 1p,(2\mu(S))^\frac 12\}(\sum^n_{i=1}\|v_{i}\|^2)^\frac 12.$$

\end{lemma}

\begin{proof}
Since $\mu$ is a propability measure on $X$, $S\subset X$ and $S\cap (-S)=\varnothing$, we have $\mu(S)\le \frac12$. Denote by $t=\mu(S)$.

Let $\{\epsilon_i\}^n_{i=1}$ be a sequence of independent Rademacher random variables in the  probability distribution $P(1)=P(-1)=\frac{1}{2}$. Given $1\leq k\leq n$, by the classical Khintchine inequality on Rademacher distributed variables (see \cite{Haagerup82}), there holds
$$
2^k A^p_p\left(\sum^k_{j=1}v^2_{i_j}\right)^\frac p2\leq \sum_{(\epsilon_{i_1},\cdots,\epsilon_{i_k})\in \{-1,1\}^k}\left|\sum^k_{j=1}\epsilon_{i_j}v_{i_j}\right|^p\leq
2^k B^p_p\left(\sum^k_{j=1}v^2_{i_j}\right)^\frac p2,
$$
where $v_i\in \mathbb{R}$, $A_p$ and $B_p$ are respectively the minimum and maximum of the
set $\{1, 2^{\frac 12-\frac 1p}, 2^{\frac 12}(\Gamma(\frac{p+1}{2})/\sqrt{\pi})^\frac 1p\}$.

Note that $\mu\{f=1\}=t=\mu\{f=-1\}$, $\mu\{f=0\}=1-2t$, and
\begin{equation}\label{eq:Ip(f)}
(I_p(v,f))^p=\sum_{k=1}^n t^k(1-2t)^{n-k}\sum_{1\le i_1<\cdots<i_k\le n}\sum_{(\epsilon_{i_1},\cdots,\epsilon_{i_k})\in \{-1,1\}^k}\left\|\sum_{j=1}^k \epsilon_{i_j}v_{i_j}\right\|^p.
\end{equation}
If $\mathbb{E}$ is of Hanner cotype $(q,n)$ with $q\leq  p$, there is

\[
\begin{split}
\sum_{(\epsilon_{i_1},\cdots,\epsilon_{i_k})\in \{-1,1\}^k}\left\|\sum_{j=1}^k \epsilon_{i_j}v_{i_j}\right\|^p
&\geq 2^{k-\frac {kp}{q}}\left(\sum_{(\epsilon_{i_1},\cdots,\epsilon_{i_k})\in \{-1,1\}^k}\left\|\sum_{j=1}^k \epsilon_{i_j}v_{i_j}\right\|^q\right)^\frac pq\\
&\geq  2^{k-\frac {kp}{q}}\left(\sum_{(\epsilon_{i_1},\cdots,\epsilon_{i_k})\in \{-1,1\}^k}\left|\sum_{j=1}^k \epsilon_{i_j}\|v_{i_j}\|\right|^q\right)^\frac pq\\
&\geq 2^{k-\frac {kp}{q}}\left(2^k A^q_q\left(\sum_{j=1}^k\|v_{i_j}\|^2\right)^{\frac q2}\right)^\frac pq = 2^k A_q^p\left(\sum_{j=1}^k\|v_{i_j}\|^2\right)^{\frac p2},
\end{split}
\]
where we used power mean inequality in the first inequality.
Hence,
\[
\begin{split}
&(I_p(v,f))^p\geq A_q^p\sum_{k=1}^n (2t)^k(1-2t)^{n-k}\sum_{1\le i_1<\cdots<i_k\le n}\left(\sum_{j=1}^k \|v_{i_j}\|^2\right)^\frac p2.
\end{split}
\]
By Lemma \ref{lemma:P(n,k)} and Jensen inequality, $I_p(v,f)$ is larger than or equal to 
$$
\begin{cases}
A_q\left(\sum_{i=1}^n \|v_{i}\|^2\right)^\frac 12\left(\sum_{k=1}^n (2t)^k(1-2t)^{n-k} {n \choose k} \frac kn\right)^\frac 1p=(2t)^\frac 1p A_q\left(\sum_{i=1}^n \|v_{i}\|^2\right)^\frac 12,&\text{ if }1\leq p \leq 2,\\
A_q\left(\sum_{i=1}^n \|v_{i}\|^2\right)^\frac 12\left(\sum_{k=1}^n (2t)^k(1-2t)^{n-k} {n \choose k} (\frac kn)^\frac p2\right)^\frac 1p\geq (2t)^\frac 12 A_q\left(\sum_{i=1}^n \|v_{i}\|^2\right)^\frac 12,&\text{ if }p \geq 2.
\end{cases}
$$
Therefore, $I_p(v,f)\geq A_q\min\{(2t)^\frac1p,(2t)^\frac12\}(\sum_{i=1}^n \|v_{i}\|^2)^{\frac{1}{2}}$.
For (2), the proof is similar 
and thus we omit it. 
\end{proof}

\begin{defn}
Let
$$G=\left\{T:X\to X\left| \int_X f(T(x))dx= \int_X f(x)dx,~\forall f\in L^1(X)\right.\right\}$$
be the collection of measure preserving transformations on $X$. 

Let $G_{odd}=\{T\in G: T(-x)=-T(x)\;,\forall x\in X\text{ a.e.}\}$ be the set of odd measure preserving transformations on $X$.
\end{defn}
\begin{lemma}\label{lem:minor}
For any $T\in G$, the map $\mathcal{T}:X^n\rightarrow X^n$ defined by
$\mathcal{T}(x_1,\cdots,x_n):=(T(x_1),\cdots,T(x_n))$ is a measure preserving transformation on $X^n$.
\end{lemma}

The proof of Lemma \ref{lem:minor} is very elementary and thus we omit it.

\begin{lemma}\label{lem:main}
The function $I_p(v,\cdot)$ possesses the following properties:
\begin{itemize}
\item[({I}1)] $I_p(v,kf)=kI_p(v,f)$ for any $k\ge0$ and $f\in L^p(X)$ (1-homogeneousity);
\item[({I}2)] $I_p(v,f+g)\le I_p(v,f)+I_p(v,g)$, $\forall f,g\in L^p(X)$ (sub-additivity); 
\item[({I}3)] $I_p(v,-f)= I_p(v,f)$ for any $f\in L^p(X)$ (even);
\item[({I}4)] $I_p(v,\cdot)$ is continuous on $L^p(X)$ (continuity);
\item[({I}5)] $I_p(v,\cdot)$ is convex on $L^p(X)$ (convexity); 
\item[({I}6)] $I_p(v,f\circ T)=I_p(v,f)$ for any $T\in G$ and $f \in L^p(X)$ (invariant);
\item[({I}7)] $I_p(v,f)\ge I_p(v,g)$ holds for odd functions $f$ and $g$, if there exists $T_1 ,T_2\in G_{odd}$ such that $|f\circ T_1(x)|\ge |g\circ T_2(x)|$ a.e. (monotonicity). 
\end{itemize}
\end{lemma}

\begin{proof}
The facts (I1),(I2),(I3) are easy. 
(I5) is a direct consequence of (I1) and (I2). While (I6) can be directly obtained by Lemma \ref{lem:minor}. For (I4), combining the triangle inequality and Young-Minkowski inequality, we have
\[
\begin{split}
\left(\int_{X^n}\left\|\sum_{i=1}^n v_{i}f(x_i)\right\|^pdx_1\cdots dx_n\right)^{\frac1p}
&\leq  \left(\int_{X^n}\left(\sum_{i=1}^n \|v_{i}f(x_i)\|\right)^pdx_1\cdots dx_n\right)^{\frac1p}\\
&\leq \sum_{i=1}^n \left(\int_{X^n}\|v_{i}f(x_i)\|^pdx_1\cdots dx_n\right)^{\frac1p} =\|f\|_p\sum^n_{i=1}\|v_i\|.
\end{split}
\]
Together with (I2) and (I3), there is
$$|I_p(v,f)-I_p(v,g)|\le I_p(v,f-g)\leq \|f-g\|_p\sum^n_{i=1}\|v_i\|$$
which implies that  $I_p(v,\cdot)$ is a continuous functional on $L^p(X,\mu)$ and then on $L^\infty(X,\mu)$.

Now we only need to prove (I7).

\begin{itemize}
\item[{Claim} 1.] If $\supp f\cap \supp \varphi=\varnothing$, then $I_p(v,f+\varphi)\ge I_p(v,f)$.

Proof of Claim 1. In fact, taking
 $$T(x)=\begin{cases}
 x,& \text{ if } x\not\in\supp \varphi, \\
 -x,& \text{ if }x\in \supp \varphi,
 \end{cases}$$
then for any $A=A_1\cup A_2$ with $A_1\subset\supp \varphi$ and $A_2\subset{(\supp \varphi)}^c $, we have
 $$\mu (T^{-1}A)=\mu(-A_1)+\mu(A_2)=\mu(A_1)+\mu(A_2)=\mu(A).$$
Hence, $T\in G$, and also $(f+\varphi)\circ T=f-\varphi$.
Thus, (I6) implies $I_p(v,f+\varphi)=I_p(v,f-\varphi)$. Together with $I_p(v,f+\varphi)+I_p(v,f-\varphi)\ge I_p(v,2f)$, one gets $I_p(v,f+\varphi)\ge I_p(v,f)$.

\item[{Claim} 2.] Let $g_a(x)=a(1_Y-1_{-Y})+g(x)$, where $g$ is an odd function with $\supp g\cap (Y\cup(-Y))=\varnothing$ and $a\ge 0$. Then the function $Z(a):=I_p(v,g_a)$ is increasing on $[0,+\infty)$.

Proof of Claim 2. Applying Claim 1 to $I_p(v,g_a)$ by taking $f=a(1_Y-1_{-Y})$ and $\varphi=g$, we immediately obtain $I_p(v,g_a)\ge I_p(v,g)$, i.e., $ Z(a)\ge Z(0)$, $\forall a\ge 0$.
Moreover, noting that $tg_a+(1-t)g_b=(ta+(1-t)b)(1_Y-1_{-Y})+g(x)=g_{ta+(1-t)b}(x)$, we may apply (I5) to obtain $tZ(a)+(1-t)Z(b)\ge Z(ta+(1-t)b)$. For any $a>b>0$, we have
$$Z(a)=\frac ba Z(a)+(1-\frac ba)Z(a) \ge \frac ba Z(a)+(1-\frac ba)Z(0)\ge Z(\frac ba a+(1-\frac ba)0)=Z(b).$$
Hence, Claim 2 is proved.

\item[{Claim} 3.] For simple functions $f(x)=\sum_{i=1}^n a_i(1_{X_i}-1_{-X_i})$ and $h(x)=\sum_{i=1}^n b_i(1_{X_i}-1_{-X_i})$ with $a_i\ge b_i\ge0$, in which $\{ X_i\cup (-X_i)\}^{n}_{i=1}$ are pairwise disjoint, there is $I_p(v,f)\ge I_p(v,h)$.

Proof of Claim 3. Applying Claim 2 to $g(c_1,c_2,\cdots,c_n):=\sum_{i=1}^n c_i(1_{X_i}-1_{-X_i})$, we have
\[
\begin{split}
I_p(v,f)&=I_p(v,g(a_1,a_2,\cdots,a_n))\\
&\ge I_p(v,g(b_1,a_2,\cdots,a_n))\\
&\ge I_p(v,g(b_1,b_2,\cdots,a_n))\\
& \cdots\\
&\ge I_p(v,g(b_1,b_2,\cdots,b_n))=I_p(v,h).
\end{split}
\]

\item[{Claim} 4.]  If two odd functions $f$ and $h$ satisfy $|f(x)|\ge |h(x)|$ and $\pm h(x)>0 \Rightarrow \pm f(x)>0$ a.e., then for any $\epsilon>0$, there exist simple functions $f_\epsilon$ and $h_\epsilon$ of the forms 
    $f_\epsilon(x)=\sum_{i=1}^n a_i(1_{X_i}-1_{-X_i})$ and $h_\epsilon(x)=\sum_{i=1}^n b_i(1_{X_i}-1_{-X_i})$ with $a_i\ge b_i\ge0$, such that $\|f-f_\epsilon\|_p<\epsilon$ and $\|h-h_\epsilon\|_p<\epsilon$. 

Proof of Claim 4.  Denote by $f^+(x)=f(x)1_{f(x)> 0}$. We can take $\widetilde{f}_\epsilon^+=\sum_{i=1}^m d_i 1_{Y_i}$ with $d_i>0$ for $1\leq i\leq m$ and $\|f^+-\widetilde{f}^+_\epsilon\|_p<\epsilon/2$.  Thus, the odd simple function $\widetilde{f}_\epsilon:=\sum_{i=1}^m d_i (1_{Y_i}-1_{-Y_i})$ satisfies\footnote{In fact, By the measure theory, an odd integrable function $f:X\to \mathbb{R}$ can be approximated by $f_n=\sum_{i=1}^{n^2}\frac{i}{n}(1_{X_i}-1_{-X_i})$, where $X_i=f^{-1}(\frac{i-1}{n},\frac{i}{n}]$, $i=1,\cdots,n^2$.}  $\|f-\widetilde{f}_\epsilon\|_p<\epsilon$. Moreover, without loss of generality, we assume 
$Y_i\subset \{x|f(x)>0\}$ and $\supp f=\cup^{m}_{i=1}(Y_i\cup -Y_i)$, $\forall i$. Similarly, there exists $\widetilde{h}_\epsilon:=\sum_{i=1}^{m'} d_i' (1_{Y_i'}-1_{-Y_i'})$ such that $\|h-\widetilde{h}_\epsilon\|_p<\epsilon$. Now, we take a refinement $\{X_i\}_{i=1}^n$ of both $\{Y_i\}_{i=1}^m$ and $\{Y_i'\}_{i=1}^{m'}$,  
and then we construct the  functions  $f_\epsilon(x)=\sum_{i=1}^n a_i(1_{X_i}-1_{-X_i})$ and $h_\epsilon(x)=\sum_{i=1}^n b_i(1_{X_i}-1_{-X_i})$ satisfying 
$\int_{X_i} |f(x)-a_i|^pdx = \min\limits_{c\in \mathbb{R}}\int_{X_i} |f(x)-c|^pdx$ and $\int_{X_i} |h(x)-b_i|^pdx = \min\limits_{c\in \mathbb{R}}\int_{X_i} |h(x)-c|^pdx$. 
Now we prove that $\|f-f_\epsilon\|_p\le\|f-\widetilde{f}_\epsilon\|_p<\epsilon$ and
$\|h-h_\epsilon\|_p\le\|h-\widetilde{h}_\epsilon\|_p<\epsilon$.
In fact, according to the above construction, there is
\[
\begin{split}
&\|f-f_\epsilon\|_p=(\int_X|f-f_\epsilon|^pdx)^\frac1p\\
=&\left(\sum^n_{i=1}\left(\int_{X_i}|f-f_\epsilon|^pdx+
\int_{-X_i}|f-f_\epsilon|^pdx\right)\right)^\frac1p\\
=&\left(\sum^n_{i=1}\left(\int_{X_i}|f-a_i|^pdx+
\int_{-X_i}|f+a_i|^pdx\right)\right)^\frac1p\\
\leq& \left(\sum^n_{i=1}\left(\int_{X_i}|f-d_i|^pdx+
\int_{-X_i}|f+d_i|^pdx\right)\right)^\frac1p
=\|f-\widetilde{f}_\epsilon\|_p<\epsilon.
\end{split}
\]
Similarly, we have $\|h-h_\epsilon\|_p\le\|h-\widetilde{h}_\epsilon\|_p<\epsilon$.

Next, we only need to prove $a_i\ge b_i\ge0$.
Since $f|_{X_i}\geq 0$, the function $c\mapsto\int_{X_i} |f(x)-c|^pdx$ is decreasing on $(-\infty,0)$. Thus the minimum of $\int_{X_i} |f(x)-c|^pdx$ must arrive at some $c\geq 0$. That means  $a_i\geq 0$. Similarly, $b_i\geq 0$.

Note that $a_i$ satisfies $\int_{X_i} |f(x)-a_i|^pdx = \min\limits_{c\in \mathbb{R}}\int_{X_i} |f(x)-c|^pdx$ if and only if
 $$\int_{\{x|x\in X_i, f(x)>a_i \}}(f(x)-a_i)^{p-1}dx=\int_{\{x|x\in X_i,f(x)<a_i  \}}(a_i-f(x))^{p-1}dx.$$
If $b_i> a_i$, then by the coarea formula,
\begin{align*}
&\int_{\{x\in X_i|h(x)>b_i\}}(h(x)-b_i)^{p-1}dx\\
=&\int^\infty_{0}\mu\{x\in X_i|h(x)> b_i,(h(x)-b_i)^{p-1}> t\}dt\\
\leq&\int^\infty_{0}\mu\{x\in X_i|h(x)> b_i,(f(x)-b_i)^{p-1}> t\}dt\\
=&\int_{\{x\in X_i|h(x)>b_i\}}(f(x)-b_i)^{p-1}dx
\leq \int_{\{x\in X_i|f(x)>b_i\}}(f(x)-b_i)^{p-1}dx\\
<& \int_{\{x\in X_i|f(x)>a_i\}}(f(x)-a_i)^{p-1}dx
=\int_{\{x\in X_i|f(x)<a_i\}}(a_i-f(x))^{p-1}dx\\
<&\int_{\{x\in X_i|f(x)<b_i\}}(b_i-f(x))^{p-1}dx
\leq\int_{\{x\in X_i|h(x)<b_i\}}(b_i-h(x))^{p-1}dx,
\end{align*}
which leads to a contradiction. Thus $b_i\leq a_i$.
\item[{Claim} 5.]  For odd functions $f$ and $h$ with the properties that $|f(x)|\ge |h(x)|$ and $\pm h(x)>0\Rightarrow\pm f(x)>0$ a.e., we have $I_p(v,f)\ge I_p(v,h)$.

Proof of Claim 5.  By Claim 3, we have $I_p(v,f_\epsilon)\ge I_p(v,h_\epsilon)$. Combining (4) and Claim 4, $I_p(v,f_\epsilon)\le I_p(v,f)+\epsilon'$ and $I_p(v,h_\epsilon)\ge I_p(v,h)-\epsilon'$, which means that $I_p(v,f)\ge I_p(v,h)-2\epsilon'$. By the arbitrariness of $\epsilon'>0$, we derive $I_p(v,f)\ge I_p(v,h)$.
\end{itemize}

Finally, we turn to prove 
(I7). Since $f$ and $g$ are odd, we may assume that $|f(x)|\ge |g(x)|$ a.e. by (I6). Now, let   $$T(x)=\begin{cases}
 -x,& \text{ if } g(x)f(x)<0, \\
 x,& \text{ otherwise; }
 \end{cases}$$
 and let $h(x)=g(T(x))$. Then, $h(x)$ is odd with $|f(x)|\ge |h(x)|$, and $\pm h(x)>0\Rightarrow\pm f(x)>0$ a.e.. It follows from (I6) and Claim 5 that $I_p(v,g)=I_p(v,h)\le I_p(v,f)$.
\end{proof}


\subsection{ Proof of Theorem \ref{Khint-norm1} }
\label{subsec:proof}

Based on Lemma \ref{lem:main}, we can derive the two-side bounds of $I_p(v,f)$ in the following way.

\begin{pro}\label{pro:twoside-control-pre}
For any nonzero odd  function $f\in L^\infty(X)$, there exist positive numbers $a$ and $b$ as well as measurable subsets $A$ and $B$ with positive measures such that $I_p(v,a(1_A-1_{-A}))\ge I_p(v,f)\ge I_p(v,b(1_B-1_{-B}))$.
\end{pro}

\begin{proof}
Let $A=\{x\in X|f(x)>0\}$ and $a=\mathop{\mathrm{ess\,sup}}\limits_{x\in X} |f(x)|$.  Then by Lemma \ref{lem:main} (I7), $I_p(v,a(1_A-1_{-A}))\ge I_p(v,f)$.

Take $b>0$ such that $\mu(\{x\in X|f(x)>b\})>0$, and let $B=\{x\in X|f(x)>b\}$. According to Lemma \ref{lem:main}, we have $I_p(v,f)\ge I_p(v,b(1_B-1_{-B}))$.
\end{proof}

Next we determine the precise upper and lower bounds of $I_p(v,f)$.

\begin{lemma}\label{sharpbyconvex}For any $T_i\in G_{odd}$ and $\lambda_i\ge0$ with $\sum_{i=1}^m \lambda_i=1$, we have
$$I_p(v,f)\ge I_p(v,\sum_{i=1}^m \lambda_if\circ T_i).$$
\end{lemma}
\begin{proof}
By (I5) and (I6) of Lemma \ref{lem:main}, $$I_p(v,f)=\sum_{i=1}^m \lambda_iI_p(v,f\circ T_i) \ge I_p(v,\sum_{i=1}^m \lambda_if\circ T_i).$$
\end{proof}

\begin{lemma}\label{sharped} For $A= \{x\in X|f(x)>a\}$ with some $a\ge 0$,  
denoting by $h_A=\frac{\|f|_{A\cup(-A)}\|_1}{2\mu(A)}(1_A-1_{-A})$, we have $I_p(v,f)\ge I_p(v,h_A)$.
\end{lemma}

\begin{proof}
Lemma \ref{lem:main} (I7) shows that $I_p(v,f)\ge I_p(v,f|_{A\cup (-A)})$. According to Birkhoff's theorem, for any $T\in G_{odd}$ with $T^{-1}(A)=A$, 
there exists $\hat{f}_{A,T}\in L^\infty(X)\subset L^1(X)$ such that
$$\frac1m \sum_{i=0}^{m-1} f|_{A\cup (-A)}(T^ix)\longrightarrow \hat{f}_{A,T}(x),\;\; m\to+\infty$$
for $x\in X$ almost everywhere. Here $\hat{f}_{A,T}\circ T=\hat{f}_{A,T}$ and $\int_A f(x)dx=\int_A \hat{f}_{A,T}(x)dx$.
Thus, by means of Lemma \ref{sharpbyconvex}, there is $I_p(v,f)\ge I_p(v,f|_{A\cup (-A)})\ge I_p(v,\hat{f}_{A,T})$.

If $T$ is further assumed to be ergodic restricted on $A$, then
$$ \hat{f}_{A,T}(x)=\frac{\|f|_{A\cup (-A)}\|_1}{2\mu(A)}(1_A(x)-1_{-A}(x))=h_A(x),\;\;\forall x\text{ a.e.}$$ 
Therefore, $I_p(v,f)\ge I_p(v,h_A)$. Note that there exists a function $\tilde{f}\in L^\infty[-\frac12,\frac12]$ possessing the same distribution with $f$ (i.e., $\nu\{\tilde{f}\le t\}=\mu\{f\le t\}$, $\forall t\in\R$, where $\nu$ is the standard  Lebesgue measure on $[-\frac12,\frac12]$). By the rearrangement of $\tilde{f}$, we may assume $\tilde{f}$ is odd and non-decreasing on $[-\frac12,\frac12]$. Accordingly, the superlevel set $A=(\delta,\frac12]$ for some $\delta\in (0,\frac12)$. Clearly, there is an ergodic transformation on the interval $A$, and thus the proof is completed.
\end{proof}

\begin{proof}[Proof of Theorem \ref{Khint-norm1}]
According to Lemma \ref{hannerAB-P}, Proposition \ref{pro:twoside-control-pre} and Lemma \ref{sharped},
there exist some positive constant $c_{f,p,q}$  such that
 $$ I_p(v,f) \ge I_p(v,b(1_B-1_{-B}))\geq c_{f,p,q}\left(\sum^n_{i=1}\|v_{i}\|^2\right)^\frac 12
$$
whenever $\mathbb{E}$ is of Hanner cotype $(q,n)$ with $q\leq p$. Similarly, there exists  $C_{f,p,q}>0$ such that
$$
 C_{f,p,q}\left(\sum^n_{i=1}\|v_{i}\|^2\right)^\frac 12\geq I_p(v,a(1_A-1_{-A}))\ge I_p(v,f)$$
whenever $\mathbb{E}$ is of Hanner type $(q,n)$ with $q\ge p$.
Therefore, we complete the proof for (1) and (2) of Theorem \ref{Khint-norm1}.

Now we provide a suitable choice of $c_{f,p,q}$ and $C_{f,p,q}$. In fact, we may take $a=\|f\|_\infty$ by the proof of Proposition \ref{pro:twoside-control-pre}, and thus we can take $C_{f,p,q}= B_q\max\{\mu (\mathrm{supp}(f))^{\frac1p},\mu (\mathrm{supp}(f))^{\frac12}\}\|f\|_\infty$ according to Lemma \ref{hannerAB-P}.

By Lemma \ref{sharped}, it is easy to verify that for any measurable set $S\subset \{f>0\}$, $$I_p(v,f)\ge \frac{\|f|_{S\cup(-S)}\|_1}{2\mu(S)}I(v,1_S-1_{-S})$$
and then it is ready to apply Lemma \ref{hannerAB-P} to get the lower bound constant $c_{f,p,q}$.

If $\mathbb{E}$ is $L^2$ space, then $\mathbb{E}$ naturally induces a standard $L^2$ norm on $\mathbb{E}^n$. According to the fact that $L^2$ space is of Hanner cotype (resp., type) $(q,n)$ for any $1\le q\le 2$ (resp., $q\ge 2$) and $n\in\mathbb{Z}_+$, the statement (3) is  proved by taking $q=p$, $S=\{f>0\}$ and $c_{f,p}= A_p\max\{\mu (\mathrm{supp}(f))^{\frac1p-1},\mu (\mathrm{supp}(f))^{-\frac12}\}\|f\|_1$.

The proof of Theorem \ref{Khint-norm1} is completed. 
\end{proof}

%

\subsection{ Bi-norm property of $I_p(\cdot,\cdot)$ }

\begin{pro}\label{pronorm}
If $f$ is non-constant a.e. (that is, $\exists c\in \mathbb{R}$ s.t. $\mu(f>c)>0$ and $\mu(f<c)>0$), then $I_p(\cdot,f)$ defines a norm on $\mathbb{E}^n$.
\end{pro}
\begin{proof}
Firstly, it is obvious that $I_p(\lambda v,f)=|\lambda|I_p(v,f)$.
Secondly, by Young-Minkowski inequality, we have the triangle inequality, i.e., $I_p(v+u,f)\le I_p(v,f)+I_p(u,f)$, $\forall v,u$.
Now, it suffices to prove 
that $I_p(v,f)=0$ implies $v=0$. Suppose the contrary that $v=(v_1,\cdots,v_n)\neq 0$.
Denote by $M=\{(c_1,\cdots,c_n)\in \mathbb{R}^n|\sum^n_{i=1}c_iv_i=0\}$.

We first show that $M$ is a linear subspace of $\mathbb{R}^n$ with dimension at most $n-1$.
Suppose the contrary, that $\dim M=n$, then $M=\mathbb{R}^n$.
Taking $c_i=1$ and $c_j=0$ for $j\in\{1,\cdots,n\}\setminus \{i\}$, we get $v_i=0$ for any $i\in\{1,\cdots, n\}$, 
which is a contradiction. So $\dim M \leq n-1$.
Furthermore, there is some $\xi=(\xi_1,\cdots,\xi_n)\in \mathbb{R}^n\setminus\{0\}$
such that $\xi \perp M$.


Define $Y=\{(x_1,\cdots, x_n)\in X^n|(f(x_1),\cdots,f(x_n))\notin M\}$.
Since $I_p(v,f)=0$, we have $(f(x_1),\cdots,f(x_n))\in M$ for $(x_1,\cdots, x_n)\in X^n$ almost everywhere. That is, $(x_1,\cdots, x_n)\in X^n\setminus Y$ a.e., which means $\hat{\mu}(Y)=0$, where $\hat{\mu}$ is the product measure on $X^n$.

Without loss of generality, we may assume that $c\sum_i\xi_i\geq 0$. For any $1\leq i \leq n$, let
$$
\theta_i=
\begin{cases}
1,& \xi_i >0,\\
-1,& \xi_i \leq0,
\end{cases}
$$
and
let $A_\theta=\{(x_1,\cdots, x_n)\in X^n|(f(x_i)-c)\theta_i> 0, 1\leq i \leq n\}$.
Then for any $(x_1,\cdots, x_n)\in A_\theta$ and for any $1\leq i \leq n$, there is $(f(x_i)-c)\xi_i\geq 0$.  Thus $\sum^n_{i=1}(f(x_i)-c)\xi_i> 0$ and $\sum^n_{i=1}f(x_i)\xi_i> \sum^n_{i=1}c\xi_i\geq 0 $.
Note that $Y \supset \{(x_1,\cdots, x_n)\in X^n|\sum^n_{i=1}f(x_i)\xi_i\neq 0\}$, which implies
$A_\theta\subset \{(x_1,\cdots, x_n)\in X^n|\sum^n_{i=1}f(x_i)\xi_i> 0\} \subset Y$. In consequence,
 $$0<\mu^m(f>c)\mu^{n-m}(f<c)=\hat{\mu}(A_\theta)\leq \hat{\mu}(Y) $$
for some $m$, which contradicts with $\hat{\mu}(Y)=0$.

%
Accordingly, $I_p(\cdot,f)$ is a norm on $\mathbb{E}^n$.
\end{proof}

\begin{pro}\label{pronorm-2}
If $\sum_{i=1}^n v_i\ne 0$, then $I_p(v,\cdot)$ defines a norm on $L^p(X)$.
\end{pro}
\begin{proof}
Similar to Proposition \ref{pronorm}, it suffices to verify that $I_p(v,f)=0$ implies $f=0$. Suppose the contrary, that there exists a nonzero $f$ satisfying $I_p(v,f)=0$. Then by $\sum_{i=1}^n v_i\ne 0$, one gets that $f$ is nonconstant and $v=(v_1,\cdots,v_n)\ne 0$. Now repeating the process of the proof of Proposition \ref{pronorm} again, 
we have $I_p(v,f)\ne 0$, which leads to a contradiction.
\end{proof}

Generally speaking, $I_p(v,f)$ is a norm of its first component $v=(v_1,\cdots,v_n)$ in $\mathbb{E}^n$ (by Proposition \ref{pronorm}) and also a norm of its second component $f$ in $L^p(X)$ (by Proposition \ref{pronorm-2}). 

From this point of view, Theorem \ref{Khint-norm1} 
 compares two norms on $\mathbb{E}^n$, the norm $\|v\|_{p,f}:=I_p(v,f)$ induced by $f$,
and the norm $\|v\|_{l^2(\mathbb{E})}:=\sqrt{\sum^n_{i=1}\|v_i\|^2}$. Moreover, the constants in the estimate are good enough. 
 The following application to the lower bound of 
 Banach-Mazur distance is an evidence for the strong performance of Theorem \ref{Khint-norm1}.
\section{An application to the estimate of  Banach-Mazur distance}
\label{sec:BM}

Given centrally symmetric convex bodies $K$ and $L$ in $\R^n$ with their center at the original point, there are normed spaces $X=(\R^n,\|\cdot\|_X)$ and $Y=(\R^n,\|\cdot\|_Y)$ such that 
$\{x\in\R^n:\|x\|_X\le 1\}= K$ and $\{y\in\R^n:\|y\|_Y\le 1\}= L$. The Banach-Mazur distance of $K$ and $L$ is defined by 
$$d(K,L)=\inf\{r\ge 1: L\subset T(K)\subset rL,T\in GL(\R^n)\},$$
where $GL(\R^n)$ denotes the general linear group on $\R^n$.  
It can be 
written as 
\begin{equation}\label{eq:BanachMazur}
d(K,L)=\inf_{T\in GL(\R^n):\|T^{-1}\|\le 1} \|T\|:=\inf_{\|T^{-1}\|\le 1}\sup\limits_{x\ne0}\frac{\|Tx\|_Y}{\|x\|_X}.
\end{equation}

We are now in a position to deal with Eq.~\eqref{eq:BanachMazur}. For this purpose, we provide the following lemma:

\begin{lemma}\label{lemma:BM} For centrally symmetric convex  bodies $K$ and $L$ in $\R^n$, we have
\begin{equation}\label{eq:BanachMazur-Ext-p}
d(K,L)=\min_{\|T^{-1}\|\le 1}\sup_{p\ge 1}\left(\int_S \|Tx\|_Y^p d\mu(x)\right)^{\frac1p},
\end{equation}
for any set $S$ with $\mathrm{Ext}(K)\subset S\subset K$ and for any probability measure $\mu$ on $S$.

If $\R^n=(\R^k)^m$ (i.e., $n=mk$) and $K=\mathop{\underbrace{\hat{K}\times\cdots\times \hat{K}}}\limits_{m\text{ times}}$ where $\hat{K}$ is a convex body in $\R^k$, then for any $p\ge 1$,
\begin{equation}\label{eq:BanachMazur-Ext-p-1}
d(K,L)\geq \min_{\|T^{-1}\|\le 1}\left(\int_{\mathrm{Ext}(\hat{K})^m} \|Tx\|_Y^p d\mu(x)\right)^{\frac1p}.
\end{equation}
\end{lemma}

\begin{proof}
Since $\|x\|_X\le 1$ determines the convex body $K$, and $\|Tx\|_Y$ is a convex function of $x$, the supremum of $\frac{\|Tx\|_Y}{\|x\|_X}$ must arrive on the extreme points of $K$.  Consequently, we arrive at
\begin{equation}\label{eq:BanachMazur-Ext}
d(K,L)=\min_{\|T^{-1}\|\le 1}\max\limits_{x\in \mathrm{Ext}(K)}\|Tx\|_Y,
\end{equation}
in which $\mathrm{Ext}(K)$ is the extreme set of $K$. Again by the convexity of the function $K\ni x\mapsto \|Tx\|_Y$, 
 and $\mathrm{Ext}(K)\subset S\subset \mathrm{conv}(\mathrm{Ext}(K))= K$, we have
$$d(K,L)=\min_{\|T^{-1}\|\le 1}\max\limits_{x\in S}\|Tx\|_Y=\min_{\|T^{-1}\|\le 1}\sup_{p\ge 1}\left(\int_S \|Tx\|_Y^p d\mu(x)\right)^{\frac1p}.$$

If $K=\hat{K}^m$, 
taking
 $S=\mathrm{Ext}(K)=\mathop{\underbrace{\mathrm{Ext}(\hat{K})\times\cdots\times \mathrm{Ext}(\hat{K})}}\limits_{m\text{ times}}=\mathrm{Ext}(\hat{K})^m$ in \eqref{eq:BanachMazur-Ext-p}, and fixing $p\ge 1$, we obtain \eqref{eq:BanachMazur-Ext-p-1}.
\end{proof}

%

%


%

In order to apply Theorem \ref{Khint-norm1} to Banach-Mazur distance, 
we shall focus on \eqref{eq:BanachMazur-Ext-p-1}. It is also worth noting that the set $S$ appearing in Lemma \ref{lemma:BM} is not necessary to be of Cartesian product form, while more general exploration is still going on.

\begin{thm}\label{thm:2}
Given the hypercube $K=\{x\in\R^n:\|x\|_\infty\le 1\}$ and the convex symmetric body $L=\{y\in\R^n:\|y\|\le 1\}$ where $\|\cdot\|$ is a norm on $\R^n$, there is $$d(K,L)\geq \sup\limits_{p\ge 1}C_{\|\cdot\|,p}A_p n^{\frac1p-\frac12}$$
where $C_{\|\cdot\|,p}:=\inf\limits_{x\ne 0}\frac{\|x\|}{\|x\|_p}\inf\limits_{y\neq 0}\frac{\|x\|_1}{\|x\|}$ depends on the norms $\|\cdot\|$ and $\|\cdot\|_p$. 

If $\|\cdot\|$ is further assumed to be of Hanner cotype $(q,n)$, then
$$
d(K,L)\geq \sup\limits_{p\ge q} A_q \widetilde{C}_{\|\cdot\|,p} \sqrt{n}
$$
where $\widetilde{C}_{\|\cdot\|,p}:=\inf\limits_{x\ne0}\frac{\|x\|}{\|x\|_p}\inf\limits_{x\ne0}\frac{\|x\|_p}{\|x\|}$.
\end{thm}
\begin{proof}Let $T=(t_{ij})_{n\times n}$ with $t_{ij}\in\R$. Denote by $S=\mathrm{Ext}(K)$ and $a_i=(t_{i1},\cdots,t_{in})$, $i=1,\cdots,n$.
Then $Tx=(a_1\cdot x,\cdots,a_n\cdot x)^\top$ and
 $\int_S \|Tx\|^p dx=\int_S \|(a_1\cdot x,\cdots,a_n\cdot x)\|^p dx$.
We note that for given $p\ge 1$ and $c:=\inf\limits_{x\ne 0}\frac{\|x\|}{\|x\|_p}$,
\begin{align*}
\left(\int_S \|Tx\|^p dx\right)^{\frac1p}&=\left(\int_S \|(a_1\cdot x,\cdots,a_n\cdot x)\|^p dx\right)^{\frac1p}
\\&\ge c \left(\int_S \|(a_1\cdot x,\cdots,a_n\cdot x)\|_p^p dx\right)^{\frac1p}
\\&= c\left(\int_S\sum_{i=1}^n|\langle a_i,x\rangle|^p dx\right)^{\frac1p}
= c\left(\sum_{i=1}^n\int_S|\langle a_i,x\rangle|^p dx\right)^{\frac1p}
\\\text{(by Theorem \ref{Khint-norm1})}\; 
&\ge c \left(\sum_{i=1}^n  (A_p\|a_i\|_2)^p \right)^{\frac1p}
\\&\ge cA_p n^{\frac1p} \sqrt[n]{\Pi_{i=1}^n  \|a_i\|_2}
\\&\ge cA_p  n^{\frac1p} \sqrt[n]{|\det(T)|} = cA_p  n^{\frac1p} |\det(T^{-1})|^{-\frac{1}{n}}
\\&\ge cA_p\tilde{c}^{-1} n^{\frac1p-\frac12}
\end{align*}
where we use the fact that $\|T^{-1}\|\le 1$ $\Rightarrow$ $\|T^{-1}x\|_\infty\le \|x\|$ $\Rightarrow$ $\|T^{-1}x\|_\infty\le \tilde{c}\|x\|_1$ $\Rightarrow$ $|\det(T^{-1})|\le (\tilde{c}\sqrt{n})^{n}$ with 
$\tilde{c}=\sup\limits_{y\neq 0}\frac{\|y\|}{\|y\|_1}$. 
So $d(K,L)\geq\frac{\inf\limits_{x\ne 0}\frac{\|x\|}{\|x\|_p}}{\sup\limits_{y\neq 0}\frac{\|y\|}{\|y\|_1}} A_p n^{\frac1p-\frac12}$. 

Let $\vec b_i=(t_{1i},\cdots,t_{ni})^\top$. Then
$Tx=\sum_{i=1}^n x_i \vec b_i$. 

By Theorem \ref{Khint-norm1}, 
since $\|\cdot\|$ is of Hanner cotype $(q,n)$ with $q\le p$, we have
\begin{equation}\label{metric}
(\int_S \|Tx\|^p dx)^{\frac1p}=\left(\int_S \|\sum_{i=1}^n x_i\vec  b_i\|^p dx\right)^{\frac1p}
\ge A_q\sqrt{\sum_{i=1}^n\|\vec b_i\|^2}.
\end{equation}

Suppose $T^{-1}=(\widehat{t}_{ij})$. 
Since $T^{-1}T=I$, we have $$\|\hat{\vec a}_i\|_{p^*}\|\vec b_i\|_p:=(\sum_{j=1}^n |\hat{t}_{ij}|^{p^*})^{\frac{1}{p^*}}(\sum_{j=1}^n |t_{ji}|^{p})^{\frac1p}\ge \sum^n_{j=1}\widehat{t}_{ij}t_{ji}=1,$$ where $\frac1p+\frac{1}{p^*}=1$.
Accordingly, the constraint $\|T^{-1}\|\le 1$ implies
$$\|\hat{\vec a}_i\|_{p^*}=\sup_{y\neq 0}\frac{\langle \hat{\vec a}_i,\vec y\rangle}{\|y\|_p} \le \sup_{y\neq 0}\frac{\max\limits_{1\le i\le n}|\langle \hat{\vec a}_i,\vec y\rangle|}{\|y\|_p} =\sup_{y\neq 0}\frac{\|T^{-1}y\|_\infty}{\|y\|_p}\leq\tilde{c}\|T^{-1}\|\le\tilde{c} $$
and then $\|\vec b_i\|_p\ge \tilde{c}^{-1}$, where $\tilde{c}=
\sup\limits_{x\ne0}\frac{\|x\|}{\|x\|_p}$.
Thus 
$$
\left(\int_S \|\sum_{i=1}^n x_i \vec b_i\|^p dx\right)^{\frac1p}
\ge A_q\sqrt{\sum_{i=1}^n\|\vec b_i\|^2}
\ge A_q\widehat{c}\sqrt{\sum_{i=1}^n\|\vec b_i\|_p^2}
 \ge A_q \widehat{c}\tilde{c}^{-1} \sqrt{n}.
$$
where $\widehat{c}=\inf\limits_{x\ne0}\frac{\|x\|}{\|x\|_p}$.
Consequently, $d(K,L)\geq A_q\inf\limits_{x\ne0}\frac{\|x\|}{\|x\|_p}\inf\limits_{x\ne0}\frac{\|x\|_p}{\|x\|} \sqrt{n}$.
\end{proof}

\begin{pro}\label{pro:lp-general}
The Banach-Mazur distance $d(l^p,\|\cdot\|)$ of a norm $\|\cdot\|$ and the $l^p$-norm on $\R^n$ has the lower bound:
$$d(l^p,\|\cdot\|)\ge \begin{cases} \sup\limits_{r\ge 1}C_{\|\cdot\|,r}A_{r} n^{\frac{1}{r}-\frac12-\frac1p},& \text{if }p\ge 2,\\
\sup\limits_{r\ge 1}C_{\|\cdot\|_*,r}A_{r} n^{\frac{1}{r}+\frac1p-\frac32} ,& \text{if }1\le p\le 2,
\\
\sup\limits_{r\ge q} A_q \widetilde{C}_{\|\cdot\|,r} n^{\frac12-\frac1p},& \text{if }p\ge 2\text{ and }\|\cdot\|\in \mathcal{H}_{q,n}^{\mathrm{cotype}}(\R),\\ 
\sup\limits_{r\ge q} A_q \widetilde{C}_{\|\cdot\|_*,r}n^{\frac1p-\frac12},& \text{if }1\le p\le 2\text{ and }\|\cdot\|_*\in \mathcal{H}_{q,n}^{\mathrm{cotype}}(\R), \end{cases}$$ 
where the constant $C$ and $\widetilde{C}$ 
are defined in Theorem \ref{thm:2}, and $\|\cdot\|_*$ is the dual norm of $\|\cdot\|$. Here for a linear space $\mathbb{E}$, we use $\mathcal{H}_{q,n}^{\mathrm{cotype}}(\mathbb{E})$ to denote the collection of all norms on $\mathbb{E}$ that is of Hanner cotype $(q,n)$. 
\end{pro}

\begin{proof}
We will use the following basic facts:
\begin{itemize}
\item $d(l^\infty,l^q)=n^{\frac1q}$ for $q\ge 2$, $d(l^1,l^p)=n^{1-\frac1p}$ for $1\le p\le 2$;
\item $d(\|\cdot\|_1,\|\cdot\|)=d(\|\cdot\|_\infty,\|\cdot\|_*)$ where $\|\cdot\|_*$ is the dual norm of $\|\cdot\|$;
\item $d(K,L)d(L,M)\ge d(K,M)$ for any centrally symmetric convex bodies $K,L$ and $M$ in $\R^n$.
\end{itemize}
Together with the above facts,  we immediately obtain that for $p\ge 2$,
$$d(l^p,\|\cdot\|)\ge \frac{d(l^\infty,\|\cdot\|)}{d(l^\infty,l^p)}=n^{-\frac1p}d(l^\infty,\|\cdot\|),$$
and for $1\le p\le 2$,
$$d(l^p,\|\cdot\|)\ge \frac{d(l^1,\|\cdot\|)}{d(l^1,l^p)}=n^{\frac1p-1}d(l^\infty,\|\cdot\|_*).$$
Therefore, the proof is completed by employing Theorem \ref{thm:2}. 
\end{proof}

Applying Proposition \ref{pro:lp-general} to $l^p$ and $l^q$ norms, we get

\begin{corollary}\label{cor:lplq-BM}
Given $K=\{x\in\R^n:\|x\|_p\le 1\}$ and $L=\{y\in\R^n:\|y\|_q\le 1\}$ with $1\le p<2<q\le \infty$, we have $d(K,L)\ge \max\{A_pn^{\frac12-\frac1q},A_{q^*}n^{\frac1p-\frac12}\}$, where $q^*$ is the H\"older conjugate of $q$ (i.e., $\frac1q+\frac{1}{q^*}=1$).
\end{corollary}


\begin{example}
 Taking $p=1$ and $q=\infty$ in Corollary \ref{cor:lplq-BM}, we have 
$d(K,L)\ge \sqrt{\frac{n}{2}}$. This confirms a surmise 
on the Banach-Mazur distance between the $n$-dimensional cube and the crosspolytope \cite{FeiXue18}.

\end{example}








\subsection*{Acknowledgements}
The authors are 
 supported by grant from the Project funded by China Postdoctoral Science Foundation (No. 2019M660829).  
We thank the 
referee for detailed suggestions on the former version of the paper, which improves the  presentation of the paper. Xin Luo appreciated the hospitality when she visited the Max Planck Institute for Mathematics in the Sciences in the winter of 2019.

\section*{Appendix }
\begin{proof}[Proof of Lemma \ref{lemma:P(n,k)}]
We first prove that if $0\le \alpha\le1$, then
\begin{equation}\label{eq:P(n,k)}
\sum_{1\le i_1<\cdots<i_k\le n}\left(\sum_{j=1}^k x_{i_j}\right)^\alpha \ge
{n-1\choose k-1}\left(\sum_{i=1}^n x_i\right)^\alpha.
\end{equation}
Denote by $P(n,k)$ the inequality \eqref{eq:P(n,k)}.
Obviously, $P(n,n)$ always holds for $n\in \mathbb{Z}_+$. Since $\sum_{i=1}^n x_i^\alpha\ge \left(\sum_{i=1}^n x_i\right)^\alpha$ holds for any $n$, $P(n,1)$ is true for any $n\in \mathbb{Z}_+$. Now, we show that $P(n-1,k-1)$ and $P(n-1,k)$ imply $P(n,k)$.

Fixing $x_1,\cdots, x_{n-1}\ge0$, let
$$\beta(x_n)=\sum_{1\le i_1<\cdots<i_k\le n}\left(\sum_{j=1}^k x_{i_j}\right)^\alpha - {n-1\choose k-1}\left(\sum_{i=1}^n x_i\right)^\alpha.$$
Then, 
\begin{align*}
\beta'(x_n)&=\alpha \sum_{1\le i_1<\cdots<i_{k-1}\le n-1,i_k=n}\left(\sum_{j=1}^k x_{i_j}\right)^{\alpha-1}-\alpha {n-1\choose k-1} \left(\sum_{i=1}^n x_i\right)^{\alpha-1}
\\& = \alpha\sum_{1\le i_1<\cdots<i_{k-1}\le n-1,i_k=n}\left(\left(\sum_{j=1}^k x_{i_j}\right)^{\alpha-1}-\left(\sum_{i=1}^n x_i\right)^{\alpha-1}\right)
\ge 0.
\end{align*}
 By inequalities $P(n-1,k-1)$ and $P(n-1,k)$, we get
\begin{align*}
\beta(0)&=\sum_{1\le i_1<\cdots<i_k\le n-1}\left(\sum_{j=1}^k x_{i_j}\right)^\alpha + \sum_{1\le i_1<\cdots<i_{k-1}\le n-1}\left(\sum_{j=1}^{k-1} x_{i_j}\right)^\alpha - {n-1\choose k-1}\left(\sum_{i=1}^{n-1} x_i\right)^\alpha
\\&\ge {n-2\choose k-1} \left(\sum_{i=1}^{n-1} x_i\right)^\alpha+ {n-2\choose k-2} \left(\sum_{i=1}^{n-1} x_i\right)^\alpha - {n-1\choose k-1}\left(\sum_{i=1}^{n-1} x_i\right)^\alpha =0
\end{align*}
according to the basic equality ${n-1\choose k-1}={n-2\choose k-1}+ {n-2\choose k-2}$. 
Thus, $\beta(x_n)\ge 0$, i.e., $P(n,k)$ is true.
Therefore, by mathematical induction, $P(n,k)$  holds for all $1\le k\le n$, $n=1,2\cdots$.

Similarly, 
if $\alpha\ge 1$, then there holds
\begin{equation}\label{eq:P(n,k)'}
\sum_{1\le i_1<\cdots<i_k\le n}\left(\sum_{j=1}^k x_{i_j}\right)^\alpha \le
{n-1\choose k-1}\left(\sum_{i=1}^n x_i\right)^\alpha.
\end{equation}

The power mean inequality
$$
\frac{1}{{n\choose k}}\sum_{1\le i_1<\cdots<i_k\le n} \sum_{j=1}^k x_{i_j}
\le \left(\frac{1}{{n\choose k}}\sum_{1\le i_1<\cdots<i_k\le n} \left(\sum_{j=1}^k x_{i_j}\right)^{\alpha_2}\right)^{1/\alpha_2}
$$
implies
\begin{align}\label{eq:alpha2}
\sum_{1\le i_1<\cdots<i_k\le n} \left(\sum_{j=1}^k x_{i_j}\right)^{\alpha_2}
&\ge {n\choose k}^{1-\alpha_2}\left(\sum_{1\le i_1<\cdots<i_k\le n} \sum_{j=1}^k x_{i_j}\right)^{\alpha_2} \notag
\\&= {n\choose k}^{1-\alpha_2}\left(\frac kn{n\choose k}\sum_{i=1}^n x_i\right)^{\alpha_2} \notag
\\&= {n\choose k}\left(\frac kn\right)^{\alpha_2}\left(\sum_{i=1}^n x_i\right)^{\alpha_2}
\end{align}
and similarly
\begin{equation}\label{eq:alpha1}
\sum_{1\le i_1<\cdots<i_k\le n} \left(\sum_{j=1}^k x_{i_j}\right)^{\alpha_1}\le {n\choose k}\left(\frac kn\right)^{\alpha_1}\left(\sum_{i=1}^n x_i\right)^{\alpha_1}
\end{equation}
where $0\le \alpha_1\le 1\le \alpha_2$. Together with inequalities \eqref{eq:P(n,k)}, \eqref{eq:P(n,k)'}, \eqref{eq:alpha2}, \eqref{eq:alpha1}, and the basic fact ${n\choose k}\frac kn={n-1\choose k-1}$, we complete the proof of \eqref{eq:lemma-main}.

Taking $(x_1,x_2,\cdots,x_n)=(1,0,\cdots,0)$, we have $$\sum_{1\le i_1<\cdots<i_k\le n}\left(\sum_{j=1}^k x_{i_j}\right)^\alpha / {n\choose k}\left(\sum_{i=1}^n x_i\right)^\alpha=\frac kn.$$
 Taking $(x_1,x_2,\cdots,x_n)=(1,1,\cdots,1)$, we obtain
 $$\sum_{1\le i_1<\cdots<i_k\le n}\left(\sum_{j=1}^k x_{i_j}\right)^\alpha / {n\choose k}\left(\sum_{i=1}^n x_i\right)^\alpha=(\frac kn)^\alpha.$$ Thus, the bounds are sharp.
\end{proof}

%
%



\end{document}